\documentclass[11pt,twoside]{article}
\usepackage{amssymb,amsmath,amsthm,amsfonts,mathrsfs,hyperref}
\usepackage{times}
\usepackage{enumerate}
\usepackage{cite,titletoc}
\usepackage[toc,page,title,titletoc,header]{appendix}

\allowdisplaybreaks

\def\cA{\mathcal A}

\def\cC{\mathcal C}

\def\cF{\mathcal F}

\def\cP{\mathcal P}

\def\cU{\mathcal U}

\def\cW{\mathcal W}
\def\cX{\mathcal X}

\def\N{\mathop{\mathbb N\kern 0pt}\nolimits}
\def\Z{\mathop{\mathbb Z\kern 0pt}\nolimits}
\def\Q{\mathop{\mathbb Q\kern 0pt}\nolimits}
\def\R{\mathop{\mathbb R\kern 0pt}\nolimits}
\def\T{\mathop{\mathbb T\kern 0pt}\nolimits}
\def\SS{\mathop{\mathbb S\kern 0pt}\nolimits}

\def\ds{\displaystyle}

\def\p{\partial}

\def\ve{\varepsilon}

\def\ls{\lesssim}

\newcommand{\w}[1]{\langle {#1} \rangle}

\hypersetup{colorlinks=true,linkcolor=blue,citecolor=red,urlcolor=cyan}

\theoremstyle{plain}
\newtheorem{theorem}{Theorem}[section]

\newtheorem{lemma}[theorem]{Lemma}
\newtheorem{corollary}[theorem]{Corollary}

\theoremstyle{definition}

\newtheorem{remark}{Remark}[section]

\numberwithin{equation}{section}

\title{Global existence and scattering of small data smooth solutions to quasilinear
wave systems on $\mathbb{R}^2\times\mathbb{T}$, II}

\author{Fei Hou$^{1,*}$ \quad Fei Tao$^{2,*}$ \quad Huicheng Yin$^{3,}$
    \footnote{Hou Fei(\texttt{fhou$@$nju.edu.cn}),  Tao Fei(\texttt{feitao$@$njupt.edu.cn}) and  Yin Huicheng(\texttt{huicheng$@$nju.edu.cn}, \texttt{05407$@$njnu.edu.cn}) are supported by the NSFC (No.~12331007, No.~12101304). In addition, Tao Fei is supported by the NSF of
    Jiangsu Provincial (No. BK20240602), and the Natural Science Research Start-up Foundation of Recruiting Talents of
    Nanjing University of Posts and Telecommunications (Grant No. NY223200);
    Yin Huicheng is supported by the National key research and development program of China (No.2020YFA0713803).}\\
    [12pt]{\small 1. School of Mathematics, Nanjing University, Nanjing, 210093, China}\\
    {\small 2. School of Science, Nanjing University of Posts and Telecommunications,}\\
    {\small Nanjing, 210023, China}\\
    {\small 3. School of Mathematical Sciences and Mathematical Institute,}\\
    {\small Nanjing Normal University, Nanjing, 210023, China}}

\begin{document}

\date{}
\maketitle
\thispagestyle{empty}

\begin{abstract}
In our previous paper [Fei Hou, Fei Tao, Huicheng Yin, Global existence and scattering of small data smooth solutions to
a class of quasilinear wave systems on $\mathbb{R}^2\times\mathbb{T}$, Preprint (2024), arXiv:2405.03242],
for the $Q_0$-type  quadratic nonlinearities, we have shown the global well-posedness and scattering properties of
small data smooth solutions to the quasilinear wave systems on $\mathbb{R}^2\times\mathbb{T}$.
In this paper, we start to solve the global existence problem for the remaining $Q_{\alpha\beta}$-type nonlinearities.
By combining these results, we have established the global well-posedness of small solutions  on $\mathbb{R}^2\times\mathbb{T}$
for the general 3-D quadratically quasilinear wave systems when the related 2-D null conditions are fulfilled.

\vskip 0.2 true cm


\qquad \quad~~space-time decay estimate, ghost weight

\vskip 0.1 true cm
\noindent
\end{abstract}

\vskip 0.1 true cm


\section{Introduction}
We continue to study the global well-posedness and scattering properties of small data smooth solutions
to the general 3-D quadratically quasilinear wave systems  on $\mathbb{R}^2\times\mathbb{T}$
when the related 2-D null conditions are fulfilled. The motivations and physical backgrounds for
studying nonlinear wave equations (or systems) on product spaces
arise from the Kaluza-Klein theory (see \cite{Kalu,Klein,Witten}) and wave propagation in infinite homogeneous
waveguides (see \cite{Ett-15,PLR-03,MSS-05}).
For the $Q_0$-type  quadratic nonlinearities,
the global solution problem on $\mathbb{R}^2\times\mathbb{T}$ has been solved in our previous paper \cite{HTY24arXiv}.
We now deal with the remaining $Q_{\alpha\beta}$-type nonlinearities, namely,
consider the following 3-D quasilinear wave systems on $\R^2\times\T$,
\begin{equation}\label{QWE}
\left\{
\begin{aligned}
&\Box u^i=G^i(\p u,\p^2u):=\sum_{j,k=1}^m\sum_{\alpha,\beta=0}^2\Big(\sum_{\gamma=0}^3C^{\alpha\beta\gamma}_{ijk}Q_{\alpha\beta}(\p_{\gamma}u^j,u^k)
+C^{\alpha\beta}_{ijk}Q_{\alpha\beta}(u^j,u^k)\Big)\\
&\hspace{3.5cm}+\sum_{j,k,l=1}^m\sum_{\alpha,\beta,\gamma=0}^3Q^{\alpha\beta\gamma}_{ijkl}\p_{\alpha}u^j\p_{\beta}u^k\p_{\gamma}u^l,\quad i=1,\cdots,m,\\
&(u,\p_tu)(0,x,y)=(u_{(0)},u_{(1)})(x,y),
\end{aligned}
\right.
\end{equation}
where $u=(u^1, ..., u^m)$, $(t,x,y)\in [0,+\infty)\times\R^2\times\T$, $x=(x_1,x_2)$, $\p=(\p_0,\p_1,\p_2,\p_3)=(\p_t,\p_{x_1},\p_{x_2},$
$\p_y)$, $\ds\Box:=\Box_{t,x,y}=\p_t^2-\Delta_x-\p_y^2$ with $\Delta_x=\p_1^2+\p_2^2$, $\ds Q_{\alpha\beta}(f,g)=\p_{\alpha}f\p_{\beta}g-\p_{\beta}f\p_{\alpha}g$ for $\alpha,\beta=0,1,2$,
$C^{\alpha\beta\gamma}_{ijk}$, $C^{\alpha\beta}_{ijk}$ and $Q^{\alpha\beta\mu}_{ijkl}$ are constants,
$u_{(r)}=(u_{(r)}^1, ..., u_{(r)}^m)$ for $r=0,1$, and $(u_{(0)},u_{(1)})$ is $2\pi$-periodic in the variable $y$.
In addition, the right hand side of the first line in \eqref{QWE} can be rewritten as
\begin{equation}\label{NL:rewrite}
\begin{split}
\sum_{\alpha,\beta=0}^2\sum_{\gamma=0}^3C^{\alpha\beta\gamma}_{ijk}Q_{\alpha\beta}(\p_{\gamma}u^j,u^k)
=\sum_{\alpha,\beta,\gamma=0}^3Q_{ijk}^{\alpha\beta\gamma}\p^2_{\alpha\beta}u^j\p_{\gamma}u^k,
\end{split}
\end{equation}
and the following symmetric conditions are imposed for $i,j,k=1,\cdots,m$,
\begin{equation}\label{sym:condition}
Q_{ijk}^{\alpha\beta\gamma}=Q_{ijk}^{\beta\alpha\gamma}=Q_{jik}^{\alpha\beta\gamma}.
\end{equation}
It is easy to check that for $i,j,k=1,\cdots,m$ and $(\xi_0,\xi_1,\xi_2)\in\{\pm1\}\times\SS$,
\begin{equation}\label{NC:rewrite}
\sum_{\alpha,\beta,\gamma=0}^2Q_{ijk}^{\alpha\beta\gamma}\xi_{\alpha}\xi_{\beta}\xi_{\gamma}\equiv0,
\end{equation}
which means that the first null condition holds for the related 2-D quadratic nonlinearities
of quasilinear wave equations. It is well-known from \cite{Klainerman, Katayama12} that when the first null condition holds,
the related 2-D quadratic  nonlinearity is only composed of the nonlinear forms $Q_0(\p u,u)$, $Q_0(u,u)$,
$Q_{\alpha\beta}(\p u,u)$ and $Q_{\alpha\beta}(u,u)$, where $Q_0(f,g)=\p_tf\p_tg-\sum_{j=1}^2\p_jf\p_jg$.
On the other hand,  for the 2-D nonlinear scalar wave equations, both the first and second null conditions
are sufficient and necessary in order to guarantee the global existence of small data smooth solutions
(see \cite{Alinhac01a,Alinhac01b}).
Therefore, for the cubic nonlinearity in \eqref{QWE}, such a second null condition is also required
in order to study the global solution problem of \eqref{QWE}:

{\it For $i,j,k,l=1,\cdots,m$ and $(\xi_0,\xi_1,\xi_2)\in\{\pm1\}\times\SS$, it holds that
\begin{equation}\label{cubic:null:condtion}
\sum_{\alpha,\beta,\mu=0}^2Q^{\alpha\beta\mu}_{ijkl}\xi_{\alpha}\xi_{\beta}\xi_{\mu}\equiv0.
\end{equation}}

As in \cite{HTY24arXiv}, we set
\begin{equation}\label{initial:data}
\begin{split}
\ve&:=\sum_{i+j\le2N+1}\|\w{x}^{4+i+j}\nabla_x^i\p_y^ju_{(0)}\|_{L^2_{x,y}}
+\sum_{i+j\le2N}\|\w{x}^{4+i+j}\nabla_x^i\p_y^ju_{(1)}\|_{L^2_{x,y}},
\end{split}
\end{equation}
where $N\ge 21$, $\nabla_x=(\p_1,\p_2)$, $\|\cdot\|_{L^2_{x,y}}:=\|\cdot\|_{L^2(\R^2\times\T)}$ and $\w{x}:=\sqrt{1+|x|^2}$.
Our main result is:
\begin{theorem}\label{thm1}
There is an $\ve_0>0$ such that when $\ve\le\ve_0$, problem \eqref{QWE} with \eqref{sym:condition} and \eqref{cubic:null:condtion} has a
global solution $u\in\bigcap\limits_{j=0}^{2N+1}C^{j}([0,\infty), H^{2N+1-j}(\R^2\times\T))$ satisfying
\begin{equation}\label{thm1:decay}
\begin{split}
&|u(t,x,y)|\le C\ve\w{t+|x|}^{-1/2}\w{t-|x|}^{-0.4},\quad|\p u(t,x,y)|\le C\ve\w{t+|x|}^{-1/2}\w{t-|x|}^{-1/2},\\
&|\p_y u(t,x,y)|\le C\ve\w{t+|x|}^{-1},
\end{split}
\end{equation}
where  $C>0$ is a generic constant independent of $\ve$.
Moreover, $u$ scatters to a free solution: there exists $(u_{(0)}^\infty,u_{(1)}^\infty)\in
\dot H^{N+2}(\R^2\times\T)\times H^{N+1}(\R^2\times\T)$ such that
\begin{equation}\label{thm1:scater}
\|\p(u(t,\cdot)-u^\infty(t,\cdot))\|_{H^{N+1}(\R^2\times\T)}\le C\ve^2(1+t)^{-1/2},
\end{equation}
where $u^\infty$ is the solution of the homogeneous linear wave equation $\Box u^\infty=0$ on $\R^2\times\T$ with suitable
initial data $(u_{(0)}^\infty,u_{(1)}^\infty)(x,y)$ being $2\pi$-periodic on $y$.
\end{theorem}

\begin{remark}
{\it It follows from the equations of \eqref{QWE} and the partial Fourier transformation on the
periodic variable $y$ that
\begin{equation}\label{reduction to WKG}
(\p_t^2-\Delta_x+|n|^2)u^i_n(t,x)=(G^i(\p u,\p^2u))_n(t,x),\quad x\in \R^2, n\in\Z, i=1,\cdots,m,
\end{equation}
where
\begin{equation*}\label{reduction to WKG-00}
u^i_n(t,x):=\frac{1}{2\pi}\int_{\T}e^{-ny\sqrt{-1}}u^i(t,x,y)dy,\quad n\in \Z.
\end{equation*}
Obviously, \eqref{reduction to WKG} is an infinite coupled 2-D nonlinear equations
whose zero-modes $u^i_0$ and non-zero modes $u^i_n$ ($n\in \Z \setminus\{0\}$)
solve the wave equations and the Klein-Gordon equations with mass $|n|$, respectively.
However, the global existence
of small data solutions to the general 2-D coupled wave-Klein-Gordon equation system with null conditions is unknown.
For the following wave-Klein-Gordon system on $\R^{1+3}$
\begin{equation*}\label{Simplied-WKG-eq}
\begin{cases}
&-\Box u=A^{\alpha\beta}\p_\alpha v \p_\beta v+Dv^2,\\
&(-\Box +1)v=u B^{\alpha\beta}\p_\alpha\p_\beta v,
\end{cases}
\end{equation*}
the global existence of small solution $(u,v)$ has been established (see \cite{Q-Wang-2015,IP19,LeF-Ma-16,Q-Wang-2020}).
In addition, there have been some interesting  results on the global well-posedness of small solutions to various finitely coupled
wave and Klein-Gordon equations in lower space dimensions under suitable restrictions on the
related nonlinearities (see \cite{Dong-JFA, Yue-M-6, Ifrim-Stingo, Stingo-Memoirs}).}
\end{remark}

\begin{remark}
{\it Note that for the general 2-D quadratic nonlinearity satisfying
 the first null condition, it is a linear combination of $Q_0(\p^{\leq 1}u,\p^{\leq 1}u)$
 and $Q_{\alpha\beta}(\p^{\leq 1}u,\p^{\leq 1}u)$ (see \cite{Klainerman, Katayama12}).
 By combining  the main result and the  argument techniques
 in  \cite{HTY24arXiv} with the corresponding ones in the present paper,
 we can prove that the analogous results to Theorem \ref{thm1} hold
 for the following more general wave systems on $\R^2\times\T$
\begin{equation*}
\left\{
\begin{aligned}
&\Box u^i=\sum_{j,k=1}^m\sum_{|a|,|b|\le1}C^{ab}_{ijk}Q_0(\p^au^j,\p^bu^k)
+\sum_{j,k=1}^m\sum_{\alpha,\beta=0}^2\Big(\sum_{\mu,\nu=0}^3C^{\alpha\beta\mu\nu}_{ijk}Q_{\alpha\beta}(\p_{\mu}u^j,\p_{\nu}u^k)\\
&\qquad+\sum_{\gamma=0}^3C^{\alpha\beta\gamma}_{ijk}Q_{\alpha\beta}(\p_{\gamma}u^j,u^k)
+C^{\alpha\beta}_{ijk}Q_{\alpha\beta}(u^j,u^k)\Big)+\sum_{j,k,l=1}^m\sum_{\alpha,\beta,\mu,\nu=0}^3Q^{\alpha\beta\mu\nu}_{ijkl}
\p^2_{\alpha\beta}u^j\p_{\mu}u^k\p_{\nu}u^l\\&\qquad
+\sum_{j,k,l=1}^m\sum_{\alpha,\beta,\gamma=0}^3Q^{\alpha\beta\gamma}_{ijkl}\p_{\alpha}u^j\p_{\beta}u^k\p_{\gamma}u^l,\qquad i=1,\cdots,m,\\
&(u,\p_tu)(0,x,y)=(u_{(0)},u_{(1)})(x,y).
\end{aligned}
\right.
\end{equation*}
when the related symmetric conditions like \eqref{sym:condition} and null conditions like \eqref{NC:rewrite} are imposed.}
\end{remark}

We now recall some results related to Theorem \ref{thm1}.
Recently, the authors in \cite{HS-21} investigate a system of quasilinear wave equations with nonlinearities
being  the combinations of quadratic null forms on $\R^3\times\T$,
which can be seen as a toy model of Einstein equations with additional compact dimensions.
It is shown in \cite{HS-21} that the  global smooth solution exists for the small and regular initial data with polynomial
decays at infinity in $\Bbb R^3$. The authors in \cite{HSW2023} demonstrate the classical global stability of the flat
Kaluza-Klein space-time $\R^{1+3}\times \mathbb{S}^1$ under small perturbations.
In \cite{HTY24}, we establish the global or almost global small data solutions of 4-D fully nonlinear wave
equations on $\R^{3}\times\T$ when the general partial null
condition is fulfilled or not, meanwhile, the global or almost global small data solutions
are also obtained for the 3-D fully nonlinear wave equations with cubic nonlinearities on $\R^{2}\times\T$.
In our previous paper \cite{HTY24arXiv}, we have solved the existence problem of global small data solutions for 3-D
quasilinear wave equation systems on $\R^{2}\times\T$ when the related quadratic nonlinearities
satisfy the $Q_0$-type null forms.

Next, we give some comments on the proof of Theorem \ref{thm1}. As pointed out in  \cite{HTY24arXiv},
the following crucial Klainerman-Sobolev inequality (see \cite{Klainerman85} and \cite{Hormander97book})
\begin{equation}\label{HCC-5}
(1+|t-r|)^{1/2}(1+t)^{(d-1)/2}|\p\phi(t,x)|\le C\ds\sum_{|I|\le [d/2]+2}\|\hat{Z}^I\p \phi(t,x)\|_{L_x^2(\Bbb R^d)}
\end{equation}
can not been used for the nonlinear problem \eqref{QWE} on $\R^2\times\T$,
where $\hat{Z}\in\{\p, x_i\p_j-x_j\p_i, x_i\p_t+t\p_i, 1\le i,j\le d, t\partial_t+\sum_{k=1}^d x_k\p_k\}$ and $r=|x|=\sqrt{x_1^2+\cdot\cdot\cdot+x_d^2}$.
To overcome this essential difficulty, in our work \cite{HTY24arXiv}, we transform the $Q_0$-type nonlinearities into
the cubic terms on  $\R^2\times\T$ by looking for some good unknowns.
However, it seems difficult for us to find suitable transformations such that the nonlinearities of $Q_{\alpha\beta}$-type
are changed into at least cubic nonlinear forms. Fortunately, the following crucial decay estimate holds (see \cite[Lem 3.3]{Sogge95})
\begin{equation}\label{Qab-Tao}
\w{t+|x|}|Q_{\alpha\beta}(f,g)|\ls|\Gamma f||\p_{t,x} g|+|\p_{t,x} f||\Gamma g|,
\end{equation}
where $\Gamma\in \{\p_{t,x},L_1,L_2,\Omega_{12}\}$ with
$L_i=x_i\p_t+t\p_i$ ($i=1,2$) and $\Omega_{12}=x_1\p_2-x_2\p_1$.
It is emphasized that  the troublesome scaling vector $S=t\partial_t+\sum_{k=1}^2 x_k\p_k$ does not appear 
in the right hand side of \eqref{Qab-Tao} (note that $S$ does not commute with the Klein-Gordon operator $\square+|n|^2$ with $n\not=0$
in \eqref{reduction to WKG}).
In this case, the vector field method can be applied for both the wave equation and the Klein-Gordon equation,
namely, for \eqref{reduction to WKG} with any $n\in\Z$.
On the other hand, we establish the key pointwise decay estimates for the solution $u$ of \eqref{QWE}
through deriving the space-time decay estimates of zero modes in Lemma \ref{lem:pw0:improv} and non-zero modes 
in Lemma \ref{YHC-01}.
Together with Alinhac's ghost weighted technique in \cite{Alinhac01a} for treating the global small solution problem
of 2-D quasilinear wave equations with null
conditions, we eventually complete the proof of Theorem \ref{thm1} by the continuation argument.

The paper is organized as follows.
In Section 2, we give  some basic lemmas and the bootstrap assumptions.
The pointwise estimates of the zero modes and the non-zero modes are obtained in Section 3.
In Section 4, based on the pointwise estimates in Section 3 and the ghost weight technique, 
we derive the required energy estimates for problem \eqref{QWE}.
In Section 5, the proof of Theorem \ref{thm1} is completed.

\vskip 0.2 true cm

\noindent \textbf{Notations:}
\begin{itemize}
  \item $\w{x}:=\sqrt{1+|x|^2}$.
  \item $\p_0:=\p_t$, $\p_1:=\p_{x_1}$, $\p_2:=\p_{x_2}$, $\p_3:=\p_y$, $\p_x:=\nabla_x=(\p_1,\p_2)$, $\p_{t,x}:=(\p_0,\p_1,\p_2)$, $\p:=(\p_t,\p_1,\p_2,\p_y)$, and $\Box_{t,x}:=\p_t^2-\Delta_x=\p_t^2-\p_1^2-\p_2^2$.
  \item For $|x|>0$, define $\bar\p_i:=\p_i+\frac{x_i}{|x|}\p_t$, $i=1,2$ and $\bar\p=(\bar\p_1,\bar\p_2)$.
 For $|x|=0$, $\bar\p$ is $\p_{t,x}$.
  \item On $\R^2\times\T$, define $L_i:=x_i\p_t+t\p_i,i=1,2$, $L:=(L_1,L_2)$, $\Omega:=\Omega_{12}:=x_1\p_2-x_2\p_1$, $\Gamma=\{\Gamma_1,\cdots,\Gamma_{6}\}:=\{\p_{t,x},L_1,L_2,\Omega_{12}\}$ and $Z=\{Z_1,\cdots,Z_{7}\}:=\Gamma\cup\{\p_y\}$.
  \item $\Gamma^a:=\Gamma_1^{a_1}\Gamma_2^{a_2}\cdots\Gamma_{6}^{a_{6}}$ for $a\in\N_0^{6}$
      and $Z^a:=Z_1^{a_1}Z_2^{a_2}\cdots Z_{7}^{a_{7}}$ for $a\in\N_0^7$.
  \item $\|\cdot\|_{L^p_x}:=\|\cdot\|_{L^p(\R^2)}$, $\|\cdot\|_{L^p_y}:=\|\cdot\|_{L^p(\T)}$ and $\|\cdot\|_{L^p_{x,y}}:=\|\cdot\|_{L^p(\R^2\times\T)}$.
  \item $P_{=0}f(t,x):=\frac{1}{2\pi}\int_{\T}f(t,x,y)dy$ and $P_{\neq0}:={\rm Id}-P_{=0}$.
  \item Auxiliary $k$-order energy of vector value function $V=(V^1, ..., V^m)$ on $\R^2\times\T$:
  \begin{equation*}
   \cX_k[V](t):=\sum_{i=1}^m\sum_{|a|\le k-1}(\|\w{t-|x|}P_{=0}\p^2Z^aV^i\|_{L^2(\R^2)}
   +\|\w{t+|x|}P_{=0}\bar\p\p Z^aV^i\|_{L^2(\R^2)}),\quad k\in\Bbb N.
  \end{equation*}
  \item $\ds f_n(t,x):=\frac{1}{2\pi}\int_{\T}e^{-ny\sqrt{-1}}f(t,x,y)dy$, $n\in\Z$, $x\in\R^2$.
  \item For $f,g\ge0$, $f\ls g$ means $f\le Cg$ for a generic constant $C>0$.
  \item $\ds|Z^{\le j}f|:=\Big(\sum_{0\le|a|\le j}|Z^af|^2\Big)^\frac12$.
  \item $\ds|\cP u|:=\Big(\sum_{i=1}^m|\cP u^i|^2\Big)^\frac12$ with $\cP\in\{\bar\p,Z\}$.
\end{itemize}

\vskip 0.1 true cm

\section{Preliminaries and bootstrap assumptions}\label{sect2}
\subsection{Some basic lemmas}\label{sect2-1}

At first, we present some properties associated with the projection operators $P_{=0}$ and $P_{\neq0}$,
which will be frequently utilized later.
\begin{lemma}
For any real valued functions $f(t,x,y)$ and $g(t,x,y)$ on $\mathbb{R}^2\times\mathbb{T}$, it holds that
\begin{equation}\label{proj:property}
\begin{split}
&\int_{\T}P_{=0}fP_{\neq0}gdy=0,\quad
\|f\|^2_{L_y^2}=\|P_{=0}f\|^2_{L_y^2}+\|P_{\neq0}f\|^2_{L_y^2},\\
&P_{=0}Zf=ZP_{=0}f,\quad P_{\neq0}Zf=ZP_{\neq0}f,\\
&P_{=0}\p_yf=\p_yP_{=0}f=0,\quad(\Gamma f)_n=\Gamma(f_n),\\
&P_{=0}(fg)=P_{=0}fP_{=0}g+P_{=0}(P_{\neq0}fP_{\neq0}g),\\
&P_{\neq0}(fg)=P_{\neq0}(P_{=0}fP_{\neq0}g)+P_{\neq0}(P_{\neq0}fP_{=0}g)+P_{\neq0}(P_{\neq0}fP_{\neq0}g),\\
&|P_{=0}f|\ls\|f\|_{L_y^\infty},\quad|P_{\neq0}f|\ls\|f\|_{L_y^\infty}.
\end{split}
\end{equation}
\end{lemma}
\begin{proof}
Since these properties can be directly verified, we omit the details here.
\end{proof}

The following lemma indicates that the nonlinearities of $Q_{\alpha\beta}$-type remain invariant
under the action of the vector fields $Z's$.

\begin{lemma}\label{lem:commut}
For any real valued functions $f(t,x,y)$ and $g(t,x,y)$ on $\mathbb{R}^2\times\mathbb{T}$, it holds
\begin{equation}\label{Commutator;Qab}
ZQ_{\alpha\beta}(f,g)=Q_{\alpha\beta}(Zf,g)+Q_{\alpha\beta}(f,Zg)+\sum_{\mu,\nu=0}^2a^{Z,\mu\nu}_{\alpha\beta}Q_{\mu\nu}(f,g),
\end{equation}
where $a^{Z,\mu\nu}_{\alpha\beta}\in\{0,\pm1\}$.
In addition, one has
\begin{equation}\label{Commutator;CC}
[Z,\p_y]=0,\qquad [Z,\p_{\alpha}]=\sum_{\beta=0}^2a^{Z,\beta}_{\alpha}\p_{\beta}\quad\text{for $\alpha=0,1,2$},
\end{equation}
where $a^{Z,\beta}_{\alpha}\in\{0,\pm1\}$.
\end{lemma}
\begin{proof}
It is easy to check that for $\alpha=0,1,2$,
\begin{equation}\label{Commutator;Qab01}
\begin{split}
[L_{i},\p_\alpha]&=-\delta_{i\alpha}\p_0-\delta_{0\alpha}\p_i,\quad i=1,2,\\
[\Omega_{12},\p_\alpha]&=-\delta_{1\alpha}\p_2+\delta_{2\alpha}\p_1,
\end{split}
\end{equation}
where $\delta_{\alpha\beta}(\alpha,\beta\in \{0,1,2\})$ is the Kronecker symbol. By \eqref{Commutator;Qab01}, we have
\begin{equation}\label{Commutator;Qab02}
\begin{split}
L_{i}Q_{\alpha\beta}(f,g)&=Q_{\alpha\beta}(L_{i}f,g)+Q_{\alpha\beta}(f,L_{i}g)
-\delta_{i\alpha}Q_{0\beta}(f,g)-\delta_{0\alpha}Q_{i\beta}(f,g)\\&
\quad-\delta_{i\beta}Q_{\alpha 0}(f,g)
-\delta_{0\beta}Q_{\alpha i}(f,g),\quad i=1,2,\\
\Omega_{12}Q_{\alpha\beta}(f,g)&=Q_{\alpha\beta}(\Omega_{12}f,g)+Q_{\alpha\beta}(f,\Omega_{12}g)
+\delta_{2\alpha}Q_{1\beta}(f,g)-\delta_{1\alpha}Q_{2\beta}(f,g)\\&
\quad+\delta_{2\beta}Q_{\alpha 1}(f,g)-\delta_{1\beta}Q_{\alpha 2}(f,g).
\end{split}
\end{equation}
Note that
\begin{equation*}
\begin{split}
\p_{\mu}Q_{\alpha\beta}(f,g)&=Q_{\alpha\beta}(\p_{\mu}f,g)+Q_{\alpha\beta}(f,\p_{\mu}g).
\end{split}
\end{equation*}
This, together with \eqref{Commutator;Qab02}, yields \eqref{Commutator;Qab}.

In addition, \eqref{Commutator;CC} follows  from \eqref{Commutator;Qab01} directly.
\end{proof}

The following lemma means that the vector fields $Z^a$ commute with the wave operator $\Box$ and the resulting nonlinearities
still fulfill the $Q_{\alpha\beta}$-type null forms for quadratic part, and the general partial null conditions for cubic part.
\begin{lemma}\label{lem:eqn:high}
Let $u(t,x,y)$ be a smooth solution of \eqref{QWE}.
Then for any multi-index $a$, $Z^au$ satisfies
\begin{equation}\label{eqn:high}
\begin{split}
\Box Z^au^i&=\sum_{j,k=1}^m\sum_{\alpha,\beta=0}^2\sum_{b+c\le a}\Big\{\sum_{\gamma=0}^3C^{\alpha\beta\gamma}_{ijk,abc}
Q_{\alpha\beta}(\p_{\gamma}Z^bu^j,Z^cu^k)
+C^{\alpha\beta}_{ijk,abc}Q_{\alpha\beta}(Z^bu^j,Z^cu^k)\Big\}\\
&+\sum_{j,k,l=1}^m\sum_{b+c+d\le a}\sum_{\alpha,\beta,\gamma=0}^3Q_{ijkl,abcd}^{\alpha\beta\gamma}\p_{\alpha}Z^bu^j\p_{\beta}Z^cu^k\p_{\gamma}Z^du^l,
\end{split}
\end{equation}
where $C_{ijk,abc}^{\alpha\beta\gamma}$, $C_{ijk,abc}^{\alpha\beta}$ and $Q_{ijkl,abcd}^{\alpha\beta\gamma}$ are constants.
In addition, $C_{ijk,aa0}^{\alpha\beta\gamma}=C_{ijk}^{\alpha\beta\gamma}$ and for any $\xi=(\xi_0,\xi_1,\xi_2)\in\{\pm1,\SS^1\}$,
\begin{equation}\label{null:high}
\sum_{\alpha,\beta,\gamma=0}^2 Q_{ijkl,abcd}^{\alpha\beta\gamma}\xi_\alpha\xi_\beta\xi_\gamma\equiv0.
\end{equation}
\end{lemma}
\begin{proof}
Lemma \ref{lem:eqn:high} follows directly from the fact of $[Z,\Box]=0$, Lemma \ref{lem:commut} and Lemma 6.6.5 of \cite{Hormander97book}.
\end{proof}

Under the null conditions, we now illustrate some basic estimates which will play important roles
in proving the global existence of classical solutions to problem \eqref{QWE}.
\begin{lemma}\label{lem:null}
Suppose that the constants $N^{\alpha\beta\gamma}$ satisfy that for any $\xi=(\xi_0,\xi_1,\xi_2)\in\{\pm1,\SS^1\}$,
\begin{equation*}
\sum_{\alpha,\beta,\gamma=0}^2N^{\alpha\beta\gamma}\xi_\alpha\xi_\beta\xi_\gamma\equiv0.
\end{equation*}
Then for smooth functions $f,g$ and $h$ on $\R^2\times\T$, it holds that
\begin{equation}\label{null:structure}
\begin{split}
\w{t+|x|}|Q_{\alpha\beta}(f,g)|&\ls|\Gamma f||\p_{t,x} g|+|\p_{t,x} f||\Gamma g|,\\
|Q_{\alpha\beta}(f,g)|&\ls|\bar\p f||\p g|+|\p f||\bar\p g|,\\
\Big|\sum_{\alpha,\beta,\gamma=0}^2N^{\alpha\beta\gamma}\p^2_{\alpha\beta}f\p_{\gamma}g\Big|
&\ls|\bar\p\p f||\p g|+|\p^2 f||\bar\p g|,\\
\Big|\sum_{\alpha,\beta,\gamma=0}^2N^{\alpha\beta\gamma}\p_{\alpha}f\p_{\beta}g\p_{\gamma}h\Big|
&\ls|\bar\p f||\p g||\p h|+|\p f||\bar\p g||\p h|+|\p f||\p g||\bar\p h|,
\end{split}
\end{equation}
where $\ds|\bar\p f|:=\Big(\sum_{i=1}^2|\bar\p_if|^2\Big)^{1/2}$.
\end{lemma}
\begin{proof}
For the proof of the first inequality in \eqref{null:structure}, one can
see \cite[Lem 3.3]{Sogge95}. The proofs on the second to fourth inequalities in \eqref{null:structure}
are analogous to those in Section 9.1 of \cite{Alinhac:book} and \cite[Lem 2.2]{HouYin20jde}, respectively,
we omit the details here.
\end{proof}

Finally, we list some inequalities including the Poincar\'{e} inequality and Hardy inequality
which are appeared in Lemma 2.6 of \cite{HTY24arXiv}.
\begin{lemma}
For any function $f(y)$, $y\in\T$, it holds that
\begin{equation}\label{Poincare:ineq}
\|P_{\neq0}f\|_{L^2(\T)}=\|f-\bar f\|_{L^2(\T)}\ls\|\p_yf\|_{L^2(\T)},\quad\bar f:=\frac{1}{2\pi}\int_{\T}f(y)dy.
\end{equation}
For any function $g(t,x)\in L^2_x(\Bbb R^2)$, $x\in\R^2$, one has
\begin{equation}\label{mod:Hardy:ineq}
\Big\|\frac{g\chi}{\w{|x|-t}}\Big\|_{L^2(\R^2)}\ls \w{t}^{1/2}\ln(2+t)\|\nabla_xg\|_{L^2(\R^2)},
\end{equation}
where $\chi=\chi(\frac{|x|}{\w{t}})$, $\chi(s)\in C^\infty(\R)$, $0\le\chi(s)\le1$ and
\begin{equation}\label{cutoff:def}
\chi(s)=\left\{
\begin{aligned}
&1,\qquad s\in[1/2,2],\\
&0,\qquad s\not\in[1/3,3].
\end{aligned}
\right.
\end{equation}
\end{lemma}

\subsection{Bootstrap assumptions}\label{sect2-2}

Define the energy for problem \eqref{QWE} as
\begin{equation*}
E_k[u](t):=\sum_{i=1}^m\sum_{|a|\le k}\|\p Z^au^i\|_{L^2(\R^2\times\T)}.
\end{equation*}
We make the following bootstrap assumptions:
\begin{equation}\label{BA1}
E_{2N}[u](t)\le\ve_1(1+t)^{\ve_2}
\end{equation}
and
\begin{align}
&\sum_{|a|\le N+3}|P_{=0}Z^au(t,x,y)|\le\ve_1\w{t+|x|}^{-1/2}\w{t-|x|}^{-0.4},\label{BA2}\\
&\sum_{|a|\le N+2}|P_{=0}\p Z^au(t,x,y)|\le\ve_1\w{x}^{-1/2}\w{t-|x|}^{-1.4},\label{BA3}\\
&\sum_{|a|\le2N-16}|P_{\neq0}Z^au(t,x,y)|\le\ve_1\w{t+|x|}^{-1},\label{BA4}
\end{align}
where $N\ge21$, $\ve_1\in(\ve,1)$ will be determined later and $\ve_2=10^{-3}$ .

By \eqref{proj:property}, \eqref{BA2}-\eqref{BA4}, one can derive the following results.
\begin{corollary}
Let $u(t,x,y)$ be the solution of \eqref{QWE} and suppose that \eqref{BA2}-\eqref{BA4} hold.
Then we have
\begin{align}
&\sum_{|a|\le N+3}|Z^au|\ls\ve_1\w{t+|x|}^{-1/2}\w{t-|x|}^{-0.4},\label{BA:pw:full}\\
&\sum_{|a|\le N+2}|P_{=0}\bar\p Z^au(t,x,y)|\ls\ve_1\w{x}^{-1/2}\w{t+|x|}^{-1}\w{t-|x|}^{-0.4}.\label{pw0:good:low}
\end{align}
\end{corollary}
\begin{proof}
\eqref{BA:pw:full} can be derived from \eqref{proj:property}, \eqref{BA2} and \eqref{BA4}.

In view of \eqref{BA3} and the fact of $\w{t-|x|}\approx\w{t+|x|}$ for $|x|\le\w{t}/2$, it suffices to prove \eqref{pw0:good:low} in the
region of $|x|\ge\w{t}/2$. According to the definition of $\bar\p_i$, one obtains that for any function $f$,
\begin{equation*}
|x|\bar\p_if=|x|(\frac{x_i}{|x|}\p_t+\p_i)f=L_if+(|x|-t)\p_if.
\end{equation*}
This, together with \eqref{BA2} and \eqref{BA3}, yields
\begin{equation*}
\begin{split}
\w{x}\sum_{|a|\le N+2}|P_{=0}\bar\p Z^au(t,x,y)|
&\ls\sum_{|b|\le N+3}|P_{=0}Z^bu|+\w{t-|x|}\sum_{|a|\le N+2}|P_{=0}\p Z^au|\\
&\ls\ve_1\w{t+|x|}^{-1/2}\w{t-|x|}^{-0.4}+\ve_1\w{x}^{-1/2}\w{t-|x|}^{-1.4},
\end{split}
\end{equation*}
which derives \eqref{pw0:good:low} immediately.
\end{proof}

\section{Pointwise estimates}\label{sect33}

This section aims to improve the bootstrap assumptions \eqref{BA2} and \eqref{BA4}.

\subsection{Higher order pointwise estimates of the zero modes}

\begin{lemma}\label{lem:Sobolev-0}[Lemma 3.1 of \cite{HTY24arXiv}]
Let $u(t,x,y)$ be the smooth solution of \eqref{QWE}, then for any multi-index $a$ with $|a|\le2N-2$, it holds that
\begin{equation}\label{pw:high}
\begin{split}
\w{x}^{1/2}|P_{=0}\p Z^au(t,x,y)|&\ls E_{|a|+2}[u](t),\\
\w{x}^{1/2}|P_{\neq0}Z^au(t,x,y)|&\ls E_{|a|+2}[u](t).
\end{split}
\end{equation}
\end{lemma}

\begin{lemma}\label{lem:Sobolev}
Let $u(t,x,y)$ be the smooth solution of \eqref{QWE}, then for any multi-index $a$ with $|a|\le2N-2$, we have that
\begin{equation}\label{pw0:Sobo}
\begin{split}
&\quad\w{x}^{1/2}\w{t-|x|}^{1/2}|P_{=0}\p Z^au(t,x,y)|
+\w{x}^{1/2}\w{t+|x|}^{1/2}|P_{=0}\bar\p Z^au(t,x,y)|\\
&\ls E_{|a|+2}[u](t)+\cX_{|a|+2}[u](t),\\
&\quad\cX_k[u](t):=\sum_{i=1}^m\sum_{|b|\le k-1}(\|\w{t-|x|}P_{=0}\p^2Z^bu^i\|_{L^2(\R^2)}+\|\w{t+|x|}P_{=0}\bar\p\p Z^bu^i\|_{L^2(\R^2)}).
\end{split}
\end{equation}
\end{lemma}
\begin{proof}
This lemma follows easily from Lemma 3.7 of \cite{HTY24}, we omit details here.
\end{proof}
To apply Lemma \ref{lem:Sobolev}, it is required to derive the estimate of $\cX_{2N}[u](t)$
on the right hand side of \eqref{pw0:Sobo}.
\begin{lemma}
Let $u(t,x,y)$ be the smooth solution of \eqref{QWE}, one then has
\begin{equation}\label{aux:energy}
\cX_{2N}[u](t)\ls\ve_1\w{t}^{2\ve_2}.
\end{equation}
\end{lemma}
\begin{proof}
Recall Lemma 5.1 of \cite{HTY24}, for smooth function $f(t,x)$, $x\in\R^2$, it holds that
\begin{equation*}
\|\w{t-|x|}\p_{t,x}^2f\|_{L^2(\R^2)}+\|\w{t+|x|}\bar\p\p_{t,x}f\|_{L^2(\R^2)}
\ls\sum_{|a|\le1}\|\p_{t,x}\Gamma^af\|_{L^2(\R^2)}+\|\w{t+|x|}\Box_{t,x}f\|_{L^2(\R^2)}.
\end{equation*}
This, together with \eqref{proj:property}, yields
\begin{equation}\label{aux:energy1}
\cX_{2N}[u](t)\ls E_{2N}[u](t)+\sum_{|a|\le2N-1}\|\w{t+|x|}\Box Z^au(t,x,y)\|_{L^2(\R^2\times\T)}.
\end{equation}
By \eqref{eqn:high} and the first line of \eqref{null:structure}, we have
\begin{equation}\label{aux:energy2}
\begin{split}
\sum_{|a|\le2N-1}|\w{t+|x|}\Box Z^au|
\ls\sum_{|b|+|c|\le2N-1}(|\p^{\le1}\p Z^bu||ZZ^cu|+|Z\p^{\le1}Z^bu||\p Z^cu|)\\
+\w{t+|x|}\sum_{|b|+|c|+|d|\le2N-1}|\p Z^bu||\p Z^cu||\p Z^du|.
\end{split}
\end{equation}
For the second line in \eqref{aux:energy2}, \eqref{BA:pw:full} implies
\begin{equation}\label{aux:energy2.5}
\sum_{|b|+|c|+|d|\le2N-1}\|\w{t+|x|}|\p Z^bu||\p Z^cu||\p Z^du|\|_{L^2(\R^2\times\T)}
\ls\ve_1E_{2N}[u](t)\ls\ve_1\w{t}^{\ve_2}.
\end{equation}
It suffices  to deal with $|\p^{\le1}\p Z^bu||ZZ^cu|$ in \eqref{aux:energy2}
since the treatment on the remaining term $|Z\p^{\le1}Z^bu||\p Z^cu|$ is similar.
When $|c|\le N$ holds on the right hand side of \eqref{aux:energy2}, \eqref{BA1} and \eqref{BA:pw:full} imply that
\begin{equation}\label{aux:energy3}
\begin{split}
\sum_{\substack{|b|+|c|\le2N-1,\\|c|\le N}}\||\p^{\le1}\p Z^bu||ZZ^cu|\|_{L^2(\R^2\times\T)}
&\ls\ve_1\w{t}^{-1/2}\sum_{|b|\le2N-1}\|\p^{\le1}\p Z^bu\|_{L^2(\R^2\times\T)}\\
&\ls\ve_1\w{t}^{-1/2}E_{2N}[u](t)\ls\ve_1\w{t}^{\ve_2-1/2}.
\end{split}
\end{equation}
When $|c|\ge N+1$, we can see that $|b|\le N-2$ and
\begin{equation}\label{aux:energy4}
|\p^{\le1}\p Z^bu||ZZ^cu|\ls|P_{\neq0}\p^{\le1}\p Z^bu||ZZ^cu|+|P_{=0}\p^{\le1}\p Z^bu||ZZ^cu|.
\end{equation}
Note that for any function $f$,
\begin{equation}\label{aux:energy5}
|Zf|\ls\w{t+|x|}|\p f|.
\end{equation}
Together with \eqref{BA4}, this implies
\begin{equation}\label{aux:energy6}
\begin{split}
\sum_{\substack{|b|+|c|\le2N-1,\\|c|\ge N}}\||P_{\neq0}\p^{\le1}\p Z^bu||ZZ^cu|\|_{L^2(\R^2\times\T)}
&\ls\ve_1\sum_{|c|\le2N-1}\|\p Z^cu\|_{L^2(\R^2\times\T)}\\
&\ls\ve_1\w{t}^{\ve_2}.
\end{split}
\end{equation}
For the last term in \eqref{aux:energy4}, it can be deduced from \eqref{mod:Hardy:ineq}, \eqref{BA3} and \eqref{aux:energy5} that
\begin{equation*}
\begin{split}
&\quad\;\||P_{=0}\p^{\le1}\p Z^bu||ZZ^cu|\|_{L^2(\R^2)}\\
&\ls\|\w{x}^{-1/2}\w{t-|x|}^{-1}ZZ^cu\|_{L^2(\R^2)}\\
&\ls\w{t}^{-1/2}\Big\|\frac{\chi(\frac{|x|}{\w{t}})ZZ^cu}{\w{t-|x|}}\Big\|_{L^2(\R^2)}
+\Big\|\Big(1-\chi(\frac{|x|}{\w{t}}\Big)\w{x}^{-1/2}\w{t-|x|}^{-1}ZZ^cu\Big\|_{L^2(\R^2)}\\
&\ls\ln(2+t)\|\p ZZ^cu\|_{L^2(\R^2)}+\Big\|\Big(1-\chi(\frac{|x|}{\w{t}}\Big)\w{x}^{-1/2}\w{t-|x|}^{-1}\w{t+|x|}\p Z^cu\Big\|_{L^2(\R^2)}\\
&\ls\ln(2+t)\|\p ZZ^cu\|_{L^2(\R^2)}+\|\p Z^cu\|_{L^2(\R^2)},
\end{split}
\end{equation*}
where $\chi$ is defined by \eqref{cutoff:def}.
Thus, we have
\begin{equation}\label{aux:energy7}
\sum_{\substack{|b|+|c|\le2N-1,\\|c|\ge N}}\||P_{=0}\p^{\le1}\p Z^bu||ZZ^cu|\|_{L^2(\R^2\times\T)}
\ls\ve_1\w{t}^{\ve_2}\ln(2+t).
\end{equation}
Collecting \eqref{aux:energy1}-\eqref{aux:energy4} and \eqref{aux:energy6}-\eqref{aux:energy7} with \eqref{BA1} derives \eqref{aux:energy}.
\end{proof}

\begin{lemma}\label{lem:pw0:high}
Let $u(t,x,y)$ be the solution of \eqref{QWE} and suppose that \eqref{BA1}-\eqref{BA4} hold.
Then for any multi-index $a$ with $|a|\le2N-2$, 
\begin{equation}\label{pw0:high}
\begin{split}
&|P_{=0}Z^au(t,x,y)|+\w{x}^{1/2}\w{t-|x|}^{1/2}|P_{=0}\p Z^au(t,x,y)|\ls\ve_1\w{t}^{2\ve_2},\\
&\w{x}^{1/2}\w{t+|x|}^{1/2}|P_{=0}\bar\p Z^au(t,x,y)|\ls\ve_1\w{t}^{2\ve_2}.
\end{split}
\end{equation}
\end{lemma}
\begin{proof}
The proof of \eqref{pw0:high} is completely analogous to that of Lemma 3.4 in \cite{HTY24arXiv}, we omit it here.
\end{proof}

\subsection{Lower order pointwise estimates of the zero modes}
Based on the estimates in the last subsection, we are able to improve the pointwise estimates of the zero modes in \eqref{BA2} and \eqref{BA3}.
First, we recall two lemmas from \cite[Lemmas 4.1, 4.2]{Kubo19}.
\begin{lemma}\label{lem:improv1}
Let $w(t,x)$ be a solution of $\Box_{t,x}w=0$ with the initial data $(w,\p_tw)|_{t=0}=(f(x),g(x))$.
Then
\begin{equation*}
\begin{split}
\w{t+|x|}^{1/2}\w{t-|x|}^{1/2}|w(t,x)|&\ls\cA_{5/2,0}[f,g],\\
\w{t+|x|}^{1/2}\w{t-|x|}^{3/2}|\p w(t,x)|&\ls\cA_{7/2,1}[f,g],
\end{split}
\end{equation*}
where $\ds\cA_{\kappa,s}[f,g]:=\sum_{\tilde\Gamma\in\{\p_1,\p_2,\Omega\}}
(\sum_{|a|\le s+1}\|\w{z}^\kappa\tilde\Gamma^af(z)\|_{L^\infty}+\sum_{|a|\le s}\|\w{z}^\kappa\tilde\Gamma^ag(z)\|_{L^\infty})$.
\end{lemma}

\begin{lemma}\label{lem:improv2}
Let $w(t,x)$ be a solution of $\Box_{t,x}w=\cF(t,x)$ with zero initial data $(w,\p_tw)|_{t=0}=(0,0)$.
Then for $\rho\in(0,1/2)$ and $\kappa>0$, 
\begin{equation*}
\begin{split}
\w{t+|x|}^{1/2}\w{t-|x|}^{\rho}|w(t,x)|&\ls\sup_{0\leq s\leq t,~|z|\leq t+|x|}(\w{z}^{1/2}\cW_{1+\rho,1+\kappa}(s,z)|\cF(s,z)|),\\
\w{x}^{1/2}\w{t-|x|}^{1+\rho}|\p w(t,x)|&\ls\sum_{|a|\le1}\sup_{0\leq s\leq t,~|z|\leq t+|x|}(\w{z}^{1/2}\cW_{1+\rho+\kappa,1}(s,z)|\Gamma^a\cF(s,z)|),
\end{split}
\end{equation*}
where $\cW_{\sigma,\lambda}(t,x):=\w{t+|x|}^\sigma(\min\{\w{x},\w{t-|x|}\})^\lambda$.
\end{lemma}

\begin{lemma}\label{lem:pw0:improv}
Let $u$ be the solution of \eqref{QWE} and suppose that \eqref{BA1}-\eqref{BA4} hold.
Then we have
\begin{equation}\label{pw0:improv}
\begin{split}
\w{t+|x|}^{0.45}\sum_{|a|\le2N-4}|P_{=0}Z^au(t,x,y)|&\ls\ve_1,\\
\w{t+|x|}^{1/2}\w{t-|x|}^{0.4}\sum_{|a|\le2N-18}|P_{=0}Z^au(t,x,y)|&\ls\ve+\ve_1^2,\\
\w{x}^{1/2}\w{t-|x|}^{1.4}\sum_{|a|\le2N-19}|P_{=0}\p Z^au(t,x,y)|&\ls\ve+\ve_1^2.
\end{split}
\end{equation}
\end{lemma}
\begin{proof}
At first, we prove the first line of \eqref{pw0:improv}.
Applying Lemmas \ref{lem:improv1}, \ref{lem:improv2} with $\rho=\kappa=\ve_2$ to \eqref{eqn:high} with $|a|\le2N-4$ and $|a|\le2N-5$, 
respectively, one has
\begin{equation}\label{pw0:improv1}
\begin{split}
&\w{t+|x|}^{1/2}\sum_{|a|\le2N-4}|P_{=0}Z^au|+\w{x}^{1/2}\w{t-|x|}\sum_{|a|\le2N-5}|P_{=0}\p Z^au|\\
&\ls\ve+\sup_{0\leq s\leq t,~|z|\leq t+|x|}\sum_{|a|\le2N-4}\w{z}^{1/2}\cW_{1+2\ve_2,1}(s,z)|P_{=0}\Box Z^au(s,z)|,
\end{split}
\end{equation}
where \eqref{initial:data} has been used.
It can be deduced from \eqref{proj:property}, \eqref{eqn:high} and the first inequality of \eqref{null:structure} that
\begin{equation}\label{pw0:improv2}
\begin{split}
&\quad|P_{=0}\Box Z^au(s,z)|\\
&\ls\sum_{|b|+|c|\le|a|}(|P_{=0}Q_{\alpha\beta}(\p Z^bu,Z^cu)|+|P_{=0}Q_{\alpha\beta}(Z^bu,Z^cu)|)+\cC^a\\
&\ls\sum_{|b|+|c|\le|a|}(|Q_{\alpha\beta}(P_{=0}\p Z^bu,P_{=0}Z^cu)|+|Q_{\alpha\beta}(P_{=0}Z^bu,P_{=0}Z^cu)|\\
&\quad+\|Q_{\alpha\beta}(P_{\neq0}\p Z^bu,P_{\neq0}Z^cu)\|_{L_y^\infty}
+\|Q_{\alpha\beta}(P_{\neq0}Z^bu,P_{\neq0}Z^cu)\|_{L_y^\infty})+\cC^a\\
&\ls\w{t+|x|}^{-1}\sum_{|b|+|c|\le|a|}(|P_{=0}\p^{\le1}\p Z^bu||P_{=0}ZZ^cu|+|P_{=0}\p^{\le1}ZZ^bu||P_{=0}\p Z^cu|\\
&\quad+\||P_{\neq0}\p^{\le1}ZZ^bu||P_{\neq0}\p Z^cu|\|_{L_y^\infty}
+\||P_{\neq0}\p^{\le1}\p Z^bu||P_{\neq0}ZZ^cu|\|_{L_y^\infty})+\cC^a,
\end{split}
\end{equation}
where
\begin{equation}\label{pw0:improv2.5}
\cC^a:=\sum_{j,k,l=1}^m\sum_{|b|+|c|+|d|\le |a|}\Big|\sum_{\alpha,\beta,\gamma=0}^3Q_{ijkl,abcd}^{\alpha\beta\gamma}
P_{=0}(\p_{\alpha}Z^bu^j\p_{\beta}Z^cu^k\p_{\gamma}Z^du^l)\Big|.
\end{equation}
When $|c|\le N$ holds in the summation of \eqref{pw0:improv2}, \eqref{BA2} and \eqref{pw0:high} imply that
\begin{equation}\label{pw0:improv3}
\sum_{\substack{|b|+|c|\le|a|,\\|c|\le N}}|P_{=0}\p^{\le1}\p Z^bu||P_{=0}ZZ^cu|
\ls\ve_1^2\w{s+|z|}^{2\ve_2-1/2}\w{z}^{-1/2}\w{s-|z|}^{-0.9}.
\end{equation}
When $|c|\ge N+1$, we can see from $|b|+|c|\le|a|\le2N-4$ that $|b|\le N-5$ holds.
Then it follows from \eqref{BA3} and \eqref{pw0:high} that
\begin{equation}\label{pw0:improv4}
\sum_{\substack{|b|+|c|\le|a|,\\|c|\ge N+1}}|P_{=0}\p^{\le1}\p Z^bu||P_{=0}ZZ^cu|
\ls\ve_1^2\w{s}^{2\ve_2}\w{z}^{-1/2}\w{s-|z|}^{-1.4}.
\end{equation}
The estimate of $\ds\sum_{|b|+|c|\le2N-4}|P_{=0}\p^{\le1}ZZ^bu||P_{=0}\p Z^cu|$ is similar to \eqref{pw0:improv3} and \eqref{pw0:improv4}.

For the last line in \eqref{pw0:improv2}, \eqref{BA1}, \eqref{BA4} and \eqref{pw:high} yield that
\begin{equation}\label{pw0:improv5}
\begin{split}
&\sum_{|b|+|c|\le|a|\le2N-4}(\||P_{\neq0}\p^{\le1}ZZ^bu||P_{\neq0}\p Z^cu|\|_{L_y^\infty}
+\||P_{\neq0}\p^{\le1}\p Z^bu||P_{\neq0}ZZ^cu|\|_{L_y^\infty})\\
&\ls\ve_1^2\w{s+|z|}^{\ve_2-1}\w{z}^{-1/2}.
\end{split}
\end{equation}
Now, we focus on the estimate of the last term $\cC^a$ in \eqref{pw0:improv2}.
In view of the facts that ${\rm Id}=P_{=0}+P_{\neq0}$ and $P_{=0}\p_y=0$, \eqref{null:high} and \eqref{null:structure} imply
\begin{equation*}
\begin{split}
\cC^a&\ls\sum_{|b|+|c|+|d|\le |a|}(|P_{=0}\bar\p Z^bu||P_{=0}\p Z^cu||P_{=0}\p Z^du|\\
&\quad+\||P_{\neq0}\p Z^bu||P_{\neq0}\p Z^cu|\|_{L_y^\infty}|P_{=0}\p Z^du|+\||P_{\neq0}\p Z^bu||P_{\neq0}\p Z^cu||P_{\neq0}\p Z^du|\|_{L_y^\infty}).
\end{split}
\end{equation*}
This, together with \eqref{BA3}, \eqref{BA4}, \eqref{pw:high} and \eqref{pw0:high} implies that
\begin{equation}\label{pw0:improv5.5}
\begin{split}
\cC^a&\ls\ve_1^3\w{s}^{3\ve_2}\w{z}^{-3/2}\w{s+|z|}^{-1/2}\w{s-|z|}^{-1.4}
+\ve_1^3\w{s+|z|}^{\ve_2-1}\w{z}^{-1}\w{s-|z|}^{-1.4}\\
&\quad+\ve_1^3\w{s+|z|}^{\ve_2-2}\w{z}^{-1/2}.
\end{split}
\end{equation}
Substituting \eqref{pw0:improv2}-\eqref{pw0:improv5.5} into \eqref{pw0:improv1} leads to
\begin{equation}\label{pw0:improv6}
\begin{split}
\sum_{|a|\le2N-4}|P_{=0}Z^au|&\ls\ve_1\w{t+|x|}^{-0.45},\\
\sum_{|a|\le2N-5}|P_{=0}\p Z^au|&\ls\ve_1\w{t+|x|}^{4\ve_2}\w{x}^{-1/2}\w{t-|x|}^{-1}.
\end{split}
\end{equation}
Next, we turn to the proofs of the second and third lines in \eqref{pw0:improv}.
Applying Lemmas \ref{lem:improv1}, \ref{lem:improv2} with $\rho=0.4$ and $\kappa=\ve_2$
to \eqref{eqn:high} with $|a|\le2N-18$ and $|a|\le2N-19$, respectively, one arrives at
\begin{equation}\label{pw0:improv7}
\begin{split}
&\w{t+|x|}^{1/2}\w{t-|x|}^{0.4}\sum_{|a|\le2N-18}|P_{=0}Z^au|+\w{x}^{1/2}\w{t-|x|}^{1.4}\sum_{|a|\le2N-19}|P_{=0}\p Z^au|\\
&\ls\ve+\sup_{0\leq s\leq t,~|z|\leq t+|x|}\sum_{|a|\le2N-18}\w{z}^{1/2}\cW_{1.4+\ve_2,1}(s,z)|P_{=0}\Box Z^au(s,z)|.
\end{split}
\end{equation}
With \eqref{BA2}-\eqref{BA4} and \eqref{pw0:improv6}, the estimates \eqref{pw0:improv3}-\eqref{pw0:improv5.5} can be improved to
\begin{equation}\label{pw0:improv8}
\begin{split}
&\sum_{|b|+|c|\le|a|\le2N-18}(|P_{=0}\p^{\le1}\p Z^bu||P_{=0}ZZ^cu|+|P_{=0}\p^{\le1}ZZ^bu||P_{=0}\p Z^cu|)\\
&\ls\ve_1^2\w{s+|z|}^{2\ve_2-0.45}\w{z}^{-1/2}\w{s-|z|}^{-1.4},
\end{split}
\end{equation}
\begin{equation}\label{pw0:improv9}
\begin{split}
&\sum_{|b|+|c|\le|a|\le2N-18}(\||P_{\neq0}\p^{\le1}ZZ^bu||P_{\neq0}\p Z^cu|\|_{L_y^\infty}
+\||P_{\neq0}\p^{\le1}\p Z^bu||P_{\neq0}ZZ^cu|\|_{L_y^\infty})\\
&\ls\ve_1^2\w{s+|z|}^{-2}
\end{split}
\end{equation}
and
\begin{equation}\label{pw0:improv10}
\begin{split}
\sum_{|a|\le2N-18}\cC^a&\ls\ve_1^3\w{s}^{10\ve_2}\w{z}^{-3/2}\w{s+|z|}^{-1/2}\w{s-|z|}^{-2}\\
&+\ve_1^3\w{s+|z|}^{4\ve_2-2}\w{z}^{-1/2}\w{s-|z|}^{-1}+\ve_1^3\w{s+|z|}^{-3}.
\end{split}
\end{equation}
Collecting \eqref{pw0:improv2}, \eqref{pw0:improv7}-\eqref{pw0:improv10} yields the second and third inequalities of \eqref{pw0:improv}.
\end{proof}

\subsection{Pointwise estimates of the non-zero modes}\label{sect4}

This subsection is dedicated to improving the bootstrap assumptions \eqref{BA4}.

\begin{lemma}\label{lem:pwKG}
Let $w(t,x,y)$ be the solution of $~\Box w=\cF(t,x,y)$, it holds that
\begin{equation*}
\begin{split}
\w{t+|x|}^{0.95}|P_{\neq0}w(t,x,y)|&\ls\sup_{\tau\in[0,t]}\sum_{|a|\le5}\w{\tau}^{-1/20}
\|\w{\tau+|x'|}P_{\neq0}Z^a\cF(\tau,x',y)\|_{L^2(\R^2\times\T)}\\
&\qquad+\sum_{|a|\le7}\|\w{x'}^2Z^aw(0,x',y)\|_{L^2(\R^2\times\T)}
\end{split}
\end{equation*}
and
\begin{equation*}
\begin{split}
\w{t+|x|}|P_{\neq0}w(t,x,y)|&\ls\sup_{\tau\in[0,t]}\sum_{|a|\le5}\w{\tau}^{1/10}
\|\w{\tau+|x'|}P_{\neq0}Z^a\cF(\tau,x',y)\|_{L^2(\R^2\times\T)}\\
&\qquad+\sum_{|a|\le7}\|\w{x'}^2Z^aw(0,x',y)\|_{L^2(\R^2\times\T)}.
\end{split}
\end{equation*}
\end{lemma}
\begin{proof}
See Lemma 4.1 in \cite{HTY24arXiv}.
\end{proof}

\begin{lemma}\label{YHC-01}
Let $u(t,x,y)$ be the solution of \eqref{QWE} and suppose that \eqref{BA1}-\eqref{BA4} hold.
Then one has
\begin{equation}\label{pwKG:improv}
\sum_{|a|\le2N-16}|P_{\neq0}Z^au(t,x,y)|\ls(\ve+\ve_1^2)\w{t+|x|}^{-1}.
\end{equation}
\end{lemma}
\begin{proof}
At first, we prove
\begin{equation}\label{pwKG:improv0.5}
\sum_{|a|\le2N-10}|P_{\neq0}Z^au(t,x,y)|\ls(\ve+\ve_1^2)\w{t+|x|}^{-0.95}.
\end{equation}
By the first estimate of Lemma \ref{lem:pwKG} and \eqref{initial:data}, we have
\begin{equation}\label{pwKG:improv1}
\begin{split}
&\w{t+|x|}^{0.95}\sum_{|a|\le2N-10}|P_{\neq0}Z^au(t,x,y)|\\
&\ls\ve+\sup_{s\in[0,t]}\sum_{|a|\le2N-5}\w{s}^{-1/20}
\|\w{s+|z|}P_{\neq0}\Box Z^au(s,z,y)\|_{L^2(\R^2\times\T)}.
\end{split}
\end{equation}
For the last term in \eqref{pwKG:improv1}, it follows from \eqref{proj:property}, \eqref{eqn:high} and the first inequality in \eqref{null:structure} that
\begin{equation}\label{pwKG:improv2}
\begin{split}
&\sum_{|a|\le2N-5}\|\w{t+|x|}P_{\neq0}\Box Z^au(t,x,y)\|_{L^2_{x,y}}\\
&\ls\sum_{|b|+|c|+|d|\le2N-5}\cC_{bcd}+\sum_{|b|+|c|\le2N-5}\Big(\|\w{t+|x|}Q_{\alpha\beta}(P_{\neq0}\p^{\le1} Z^bu,P_{=0}Z^cu)\|_{L^2_{x,y}}\\
&+\|\w{t+|x|}Q_{\alpha\beta}(P_{=0}\p^{\le1} Z^bu,P_{\neq0}Z^cu)\|_{L^2_{x,y}}
+\|\w{t+|x|}Q_{\alpha\beta}(P_{\neq0}\p^{\le1} Z^bu,P_{\neq0}Z^cu)\|_{L^2_{x,y}}\Big)\\
&\ls\sum_{|b|+|c|\le2N-5}\Big(\||P_{\neq0}\p\p^{\le1}Z^bu||P_{=0}ZZ^cu|\|_{L^2_{x,y}}
+\||P_{\neq0}Z\p^{\le1} Z^bu||P_{=0}\p Z^cu|\|_{L^2_{x,y}}\\
&\quad+\||P_{=0}\p\p^{\le1}Z^bu||P_{\neq0}ZZ^cu|\|_{L^2_{x,y}}
+\||P_{=0}Z\p^{\le1} Z^bu||P_{\neq0}\p Z^cu|\|_{L^2_{x,y}}\\
&\quad+\||P_{\neq0}\p\p^{\le1}Z^bu||P_{\neq0}ZZ^cu|\|_{L^2_{x,y}}
+\||P_{\neq0}Z\p^{\le1} Z^bu||P_{\neq0}\p Z^cu|\|_{L^2_{x,y}}\Big)\\
&\quad+\sum_{|b|+|c|+|d|\le2N-5}\cC_{bcd},
\end{split}
\end{equation}
where
\begin{equation}\label{pwKG:improv2.5}
\begin{split}
\cC_{bcd}&:=\|\w{t+|x|}|P_{\neq0}\p Z^bu||P_{=0}\p Z^cu||P_{=0}\p Z^du|\|_{L^2_{x,y}}\\
&\quad+\|\w{t+|x|}|P_{\neq0}\p Z^bu||P_{\neq0}\p Z^cu||P_{=0}\p Z^du|\|_{L^2_{x,y}}\\
&\quad+\|\w{t+|x|}|P_{\neq0}\p Z^bu||P_{\neq0}\p Z^cu||P_{\neq0}\p Z^du|\|_{L^2_{x,y}}.
\end{split}
\end{equation}
By \eqref{Poincare:ineq}, \eqref{BA1}, \eqref{pw0:high} and the first inequality of \eqref{pw0:improv},
one has
\begin{equation}\label{pwKG:improv3}
\begin{split}
&\sum_{|b|+|c|\le2N-5}(\||P_{\neq0}\p\p^{\le1}Z^bu||P_{=0}ZZ^cu|\|_{L^2_{x,y}}
+\||P_{\neq0}Z\p^{\le1} Z^bu||P_{=0}\p Z^cu|\|_{L^2_{x,y}})\\
&\ls\ve_1\w{t}^{-0.45}E_{2N}[u](t)
\ls\ve_1^2\w{t}^{-0.4}.
\end{split}
\end{equation}
Analogously, we can obtain
\begin{equation}\label{pwKG:improv4}
\begin{split}
&\sum_{|b|+|c|\le2N-5}(\||P_{=0}\p\p^{\le1}Z^bu||P_{\neq0}ZZ^cu|\|_{L^2_{x,y}}
+\||P_{=0}Z\p^{\le1} Z^bu||P_{\neq0}\p Z^cu|\|_{L^2_{x,y}})\\
&\ls\ve_1^2\w{t}^{-0.4}.
\end{split}
\end{equation}
Note that $N\ge21$ ensures that $|b|\le2N-18$ or $|c|\le2N-18$. Thus, \eqref{Poincare:ineq} and \eqref{BA4} lead to
\begin{equation}\label{pwKG:improv5}
\begin{split}
&\sum_{|b|+|c|\le2N-5}(\||P_{\neq0}\p\p^{\le1}Z^bu||P_{\neq0}ZZ^cu|\|_{L^2_{x,y}}
+\||P_{\neq0}Z\p^{\le1} Z^bu||P_{\neq0}\p Z^cu|\|_{L^2_{x,y}})\\
&\ls\ve_1\w{t}^{-1}E_{2N}[u](t)
\ls\ve_1^2\w{t}^{-0.9}.
\end{split}
\end{equation}
Next, we turn to the estimate of $\cC_{bcd}$ given by \eqref{pwKG:improv2.5}. It follows from \eqref{BA3} and \eqref{BA4} that
\begin{equation}\label{pwKG:improv6}
\sum_{|b|+|c|+|d|\le2N-5}\cC_{bcd}\ls\ve_1^2E_{2N}(t)\ls\ve_1^3\w{t}^{\ve_2}.
\end{equation}
Collecting \eqref{pwKG:improv1}-\eqref{pwKG:improv6} yields \eqref{pwKG:improv0.5}.
Applying the second estimate of Lemma \ref{lem:pwKG} with \eqref{initial:data} leads to
\begin{equation}\label{pwKG:improv7}
\begin{split}
&\w{t+|x|}\sum_{|a|\le2N-16}|P_{\neq0}Z^au(t,x,y)|\\
&\ls\ve+\sup_{s\in[0,t]}\sum_{|a|\le2N-11}\w{s}^{1/10}
\|\w{s+|z|}P_{\neq0}\Box Z^au(s,z,y)\|_{L^2(\R^2\times\T)}.
\end{split}
\end{equation}
With \eqref{pwKG:improv0.5}, we can improve the estimate \eqref{pwKG:improv6} to
\begin{equation}\label{pwKG:improv8}
\sum_{|b|+|c|+|d|\le2N-11}\cC_{bcd}\ls\ve_1^3\w{t}^{-0.4}.
\end{equation}
Substituting \eqref{pwKG:improv2}-\eqref{pwKG:improv5} and \eqref{pwKG:improv8} into \eqref{pwKG:improv7} yields \eqref{pwKG:improv}.
\end{proof}

\section{Energy estimates}\label{sect5}

\begin{lemma}\label{lem:energy}
Let $u$ be the solution of \eqref{QWE} and suppose that \eqref{BA1}-\eqref{BA4} hold.
Then
\begin{equation}\label{energy}
E^2_{2N}[u](t')\ls\ve^2+\ve_1\int_0^{t'}\w{t}^{-1}E^2_{2N}[u](t)dt.
\end{equation}
\end{lemma}
\begin{proof}
Firstly, by the virtue of \eqref{NL:rewrite}, \eqref{eqn:high} can be rewritten as
\begin{equation}\label{energy1}
\begin{split}
&\Box Z^au^i=\sum_{j,k=1}^m\sum_{\alpha,\beta,\gamma=0}^3Q^{\alpha\beta\gamma}_{ijk}\p^2_{\alpha\beta}Z^au^j\p_{\gamma}u^k\\
&\quad+\sum_{j,k=1}^m\sum_{\alpha,\beta=0}^2\Big\{\sum_{\substack{b+c\le a,\\|b|<|a|}}\sum_{\gamma=0}^3
C^{\alpha\beta\gamma}_{ijk,abc}Q_{\alpha\beta}(\p_{\gamma}Z^bu^j,Z^cu^k)+\sum_{b+c\le a}C^{\alpha\beta}_{ijk,abc}Q_{\alpha\beta}(Z^bu^j,Z^cu^k)\Big\}
\\&\quad+\sum_{j,k,l=1}^m\sum_{b+c+d\le a}\sum_{\alpha,\beta,\gamma=0}^3Q_{ijkl,abcd}^{\alpha\beta\gamma}\p_{\alpha}Z^bu^j\p_{\beta}Z^cu^k\p_{\gamma}Z^du^l.
\end{split}
\end{equation}
For any multi-index $a$ with $|a|\le2N$, multiplying \eqref{energy1} by $e^q\p_tZ^au^i$ with $q=q(|x|-t)$ and $q(s)=\int_{-\infty}^s\frac{dt}{\w{t}^{1.1}}$ yields that
\begin{equation}\label{energy2}
\begin{split}
&\frac12\sum_{i=1}^m\p_t(e^q|\p Z^au^i|^2)-\sum_{i=1}^m\sum_{\gamma=1}^3\p_{\gamma}(e^q\p_tZ^au^i\p_{\gamma}Z^au^i)
+\sum_{i=1}^m\frac{e^q}{2\w{t-|x}^{1.1}}(|\bar\p Z^au^i|^2+|\p_yZ^au^i|^2)\\
&\quad=\sum_{\alpha,\beta,\gamma=0}^3e^qQ_{ijk}^{\alpha\beta\gamma}\p_tZ^au^i\p^2_{\alpha\beta}Z^au^j\p_{\gamma}u^k\\
&\quad+\sum_{\alpha,\beta=0}^2e^q\p_tZ^au^i\Big\{\sum_{\substack{b+c\le a,\\|b|<|a|}}\sum_{\gamma=0}^3
C^{\alpha\beta\gamma}_{ijk,abc}Q_{\alpha\beta}(\p_{\gamma}Z^bu^j,Z^cu^k)
+\sum_{b+c\le a}C^{\alpha\beta}_{ijk,abc}Q_{\alpha\beta}(Z^bu^j,Z^cu^k)\Big\}\\
&\quad+\sum_{j,k,l=1}^m\sum_{b+c+d\le a}\sum_{\alpha,\beta,\gamma=0}^3
Q_{ijkl,abcd}^{\alpha\beta\gamma}\p_{\alpha}Z^bu^j\p_{\beta}Z^cu^k\p_{\gamma}Z^du^l,
\end{split}
\end{equation}
where the summations over $i,j,k=1,\cdots,m$ on the right hand side of \eqref{energy2} are omitted.
By the symmetric conditions \eqref{sym:condition}, for the terms on the second line of \eqref{energy2}, we have
\begin{equation}\label{energy3}
\begin{split}
&\qquad e^qQ_{ijk}^{\alpha\beta\gamma}\p_tZ^au^i\p^2_{\alpha\beta}Z^au^j\p_{\gamma}u^k\\
&=\p_\alpha(e^qQ_{ijk}^{\alpha\beta\gamma}\p_tZ^au^i\p_{\beta}Z^au^j\p_{\gamma}u^k)
-e^qQ_{ijk}^{\alpha\beta\gamma}\p_{\alpha}q\p_tZ^au^i\p_{\beta}Z^au^j\p_{\gamma}u^k\\
&\quad-e^qQ_{ijk}^{\alpha\beta\gamma}\p_tZ^au^i\p_{\beta}Z^au^j\p^2_{\alpha\gamma}u^k
-e^qQ_{ijk}^{\alpha\beta\gamma}\p_t\p_{\alpha}Z^au^i\p_{\beta}Z^au^j\p_{\gamma}u^k\\
&=\p_\alpha(e^qQ_{ijk}^{\alpha\beta\gamma}\p_tZ^au^i\p_{\beta}Z^au^j\p_{\gamma}u^k)
-e^qQ_{ijk}^{\alpha\beta\gamma}\p_{\alpha}q\p_tZ^au^i\p_{\beta}Z^au^j\p_{\gamma}u^k\\
&\quad-e^qQ_{ijk}^{\alpha\beta\gamma}\p_tZ^au^i\p_{\beta}Z^au^j\p^2_{\alpha\gamma}u^k
-\frac{1}{2}\p_t(e^qQ_{ijk}^{\alpha\beta\gamma}\p_{\alpha}Z^au^i\p_{\beta}Z^au^j\p_{\gamma}u^k)\\
&\quad+\frac{1}{2}e^qQ_{ijk}^{\alpha\beta\gamma}\p_tq\p_{\alpha}Z^au^i\p_{\beta}Z^au^j\p_{\gamma}u^k
+\frac{1}{2}e^qQ_{ijk}^{\alpha\beta\gamma}\p_{\alpha}Z^au^i\p_{\beta}Z^au^j\p^2_{0\gamma}u^k,
\end{split}
\end{equation}
where the summation over $\alpha,\beta,\gamma=0,1,2,3$ is omitted.

By integrating \eqref{energy2}-\eqref{energy3} over $[0,t']\times\R^2\times\T$, summing over all $|a|\le2N$
and applying \eqref{BA:pw:full}, we arrive at
\begin{equation}\label{energy4}
\begin{split}
&\quad E^2_{2N}[u](t')+\sum_{|a|\le2N}\int_0^{t'}\iint_{\R^2\times\T}\frac{|\bar\p Z^au|^2+|\p_yZ^au|^2}{\w{t-|x|}^{1.1}}dxdydt\\
&\ls E^2_{2N}[u](0)+\ve_1 E^2_{2N}[u](t')+\sum_{|a|\le2N}\int_0^{t'}\iint_{\R^2\times\T}I(t,x,y)dxdydt\\
&\quad+\sum_{\substack{|b|+|c|\le|a|\le2N,\\|b|\le2N-1}}\sum_{\alpha,\beta=0}^2\sum_{\gamma=0}^3\int_0^{t'}\iint_{\R^2\times\T}
|\p_tZ^au^iQ_{\alpha\beta}(\p_{\gamma}Z^bu^j,Z^cu^k)|dxdydt\\
&\quad+\sum_{|b|+|c|\le|a|\le2N}\sum_{\alpha,\beta=0}^2\int_0^{t'}\iint_{\R^2\times\T}|\p_tZ^au^iQ_{\alpha\beta}(Z^bu^j,Z^cu^k)|dxdydt\\
&\quad+\sum_{|b|+|c|+|d|\le|a|\le2N}\sum_{\alpha,\beta,\gamma=0}^3\int_0^{t'}\iint_{\R^2\times\T}|\p_tZ^au^i
\p_{\alpha}Z^bu^j\p_{\beta}Z^cu^k\p_{\gamma}Z^du^l|dxdydt,
\end{split}
\end{equation}
where
\begin{equation}\label{energy5}
\begin{split}
&I(t,x,y):=
|\p_tZ^au^i|\Big(\Big|\sum_{\alpha,\beta,\gamma=0}^3Q_{ijk}^{\alpha\beta\gamma}\p_{\alpha}q\p_{\beta}Z^au^j\p_{\gamma}u^k\Big|
+\Big|\sum_{\alpha,\beta,\gamma=0}^3Q_{ijk}^{\alpha\beta\gamma}\p_{\beta}Z^au^j\p^2_{\alpha\gamma}u^k\Big|\Big)\\
&\quad +|\p_tq|\Big|\sum_{\alpha,\beta,\gamma=0}^3Q_{ijk}^{\alpha\beta\mu}\p_{\alpha}Z^au^i\p_{\beta}Z^au^j\p_{\gamma}u^k\Big|
+\Big|\sum_{\alpha,\beta,\gamma=0}^3Q_{ijk}^{\alpha\beta\gamma}\p_{\alpha}Z^au^i\p_{\beta}Z^au^j\p^2_{0\gamma}u^k\Big|.
\end{split}
\end{equation}
For the first summation in $I(t,x,y)$, by the facts of $1=P_{\neq0}+P_{=0}$, $\p_yP_{=0}=0$ and $\p_yq=0$, we can see that
\begin{equation}\label{energy6}
\begin{split}
&\sum_{\alpha,\beta,\gamma=0}^3Q_{ijk}^{\alpha\beta\gamma}\p_{\alpha}q\p_{\beta}Z^au^j\p_{\gamma}u^k
=\sum_{\alpha,\beta,\gamma=0}^3Q_{ijk}^{\alpha\beta\gamma}\p_{\alpha}q\p_{\beta}Z^au^jP_{\neq0}\p_{\gamma}u^k\\
&\quad+\sum_{\alpha,\gamma=0}^2Q_{ijk}^{\alpha3\gamma}\p_{\alpha}q\p_yZ^au^jP_{=0}\p_{\gamma}u^k
+\sum_{\alpha,\beta,\gamma=0}^2Q_{ijk}^{\alpha\beta\gamma}\p_{\alpha}q\p_{\beta}Z^au^jP_{=0}\p_{\gamma}u^k.
\end{split}
\end{equation}
The other terms in $I(t,x,y)$ can be analogously treated.
For the last term in \eqref{energy6}, it can be deduced from \eqref{NC:rewrite}, the third and fourth lines of \eqref{null:structure}
with the fact of $\bar\p q=0$ that
\begin{equation}\label{energy7}
\begin{split}
I(t,x,y)\ls|\p Z^au|^2(|P_{\neq0}\p^{\le1}\p u|+|P_{=0}\bar\p \p^{\le1}u|)
+|\p Z^au|(|\bar\p Z^au|+|\p_yZ^au|)|P_{=0}\p\p^{\le1} u|.
\end{split}
\end{equation}
Then from \eqref{BA3}, \eqref{BA4}, \eqref{pw0:good:low} and Young's inequality, one has
\begin{equation}\label{energy8}
\begin{split}
&\sum_{|a|\le2N}\int_0^{t'}\iint_{\R^2\times\T}I(t,x,y)dxdydt\\
&\ls\ve_1\sum_{|a|\le2N}\int_0^{t'}\iint_{\R^2\times\T}\frac{|\bar\p Z^au|^2+|\p_yZ^au|^2}{\w{t-|x|}^{1.1}}dxdydt
+\ve_1\int_0^{t'}\w{t}^{-1}E^2_{2N}[u](t)dt.
\end{split}
\end{equation}
Next, we turn to the estimates of the third and fourth lines in \eqref{energy4}.

We now treat the case of $|b|\le N$.
It follows from the second line of \eqref{null:structure} that
\begin{equation}\label{energy9}
\begin{split}
|Q_{\alpha\beta}(\p_{\gamma}Z^bu^j,Z^cu^k)|
&\ls|\bar\p\p Z^bu||\p Z^cu|+|\p^2Z^bu||\bar\p Z^cu|\\
&\ls|\bar\p\p Z^bu||\p Z^cu|+|P_{\neq0}\p^2Z^bu||\bar\p Z^cu|+|P_{=0}\p^2Z^bu||\bar\p Z^cu|\\
|Q_{\alpha\beta}(Z^bu^j,Z^cu^k)|
&\ls|\bar\p Z^bu||\p Z^cu|+|P_{\neq0}\p Z^bu||\bar\p Z^cu|+|P_{=0}\p Z^bu||\bar\p Z^cu|.
\end{split}
\end{equation}
Therefore, by using \eqref{energy9} with \eqref{BA3}, \eqref{BA4}, \eqref{pw0:good:low} and Young's inequality again, we can get that
\begin{equation}\label{energy11}
\begin{split}
&\sum_{\substack{|b|+|c|\le|a|\le2N,\\|b|\le N}}\sum_{\alpha,\beta=0}^2\int_0^{t'}\iint_{\R^2\times\T}
|\p_tZ^au^i|\Big\{\sum_{\gamma=0}^3|Q_{\alpha\beta}(\p_{\gamma}Z^bu^j,Z^cu^k)|+|Q_{\alpha\beta}(Z^bu^j,Z^cu^k)|\Big\}dxdydt\\
&\ls\ve_1\sum_{|c|\le2N}\int_0^{t'}\iint_{\R^2\times\T}\frac{|\bar\p Z^cu|^2}{\w{t-|x|}^{1.1}}dxdydt
+\ve_1\int_0^{t'}\w{t}^{-1}E^2_{2N}[u](t)dt.
\end{split}
\end{equation}
If $|b|\ge N+1$, then $|c|\le N-1$.
Instead of \eqref{energy9}, one has
\begin{equation*}
\begin{split}
|Q_{\alpha\beta}(\p_{\gamma}Z^bu^j,Z^cu^k)|
&\ls|\p^2Z^bu||\bar\p Z^cu|+|\bar\p\p Z^bu||P_{\neq0}\p Z^cu|+|\bar\p\p Z^bu||P_{=0}\p Z^cu|,\\
|Q_{\alpha\beta}(Z^bu^j,Z^cu^k)|
&\ls|\p Z^bu||\bar\p Z^cu|+|\bar\p Z^bu||P_{\neq0}\p Z^cu|+|\bar\p Z^bu||P_{=0}\p Z^cu|.
\end{split}
\end{equation*}
Analogously, 
\begin{equation}\label{energy12}
\begin{split}
&\sum_{\substack{|b|+|c|\le|a|\le2N,\\N+1\le|b|\le2N-1}}\sum_{\alpha,\beta=0}^2\sum_{\gamma=0}^3\int_0^{t'}\iint_{\R^2\times\T}
|\p_tZ^au^iQ_{\alpha\beta}(\p_{\gamma}Z^bu^j,Z^cu^k)|\\
&\quad +\sum_{\substack{|b|+|c|\le|a|\le2N,\\|b|\ge N+1}}\sum_{\alpha,\beta=0}^2\int_0^{t'}\iint_{\R^2\times\T}
|\p_tZ^au^iQ_{\alpha\beta}(Z^bu^j,Z^cu^k)|dxdydt\\
&\ls\ve_1\sum_{|c|\le2N}\int_0^{t'}\iint_{\R^2\times\T}\frac{|\bar\p Z^cu|^2}{\w{t-|x|}^{1.1}}dxdydt
+\ve_1\int_0^{t'}\w{t}^{-1}E^2_{2N}[u](t)dt.
\end{split}
\end{equation}
For the last line in \eqref{energy4}, \eqref{BA:pw:full} leads to
\begin{equation}\label{energy13}
\begin{split}
&\sum_{|b|+|c|+|d|\le|a|\le2N}\sum_{\alpha,\beta,\gamma=0}^3\int_0^{t'}\iint_{\R^2\times\T}|\p_tZ^au^i
\p_{\alpha}Z^bu^j\p_{\beta}Z^cu^k\p_{\gamma}Z^du^l|dxdydt\\
&\ls\ve_1^2\int_0^{t'}\w{t}^{-1}E^2_{2N}[u](t)dt.
\end{split}
\end{equation}
Finally, it follows from \eqref{initial:data}, \eqref{energy4}, \eqref{energy8}, \eqref{energy11},
\eqref{energy12}, \eqref{energy13} and the smallness of $\ve_1$ that \eqref{energy} holds.
\end{proof}

\section{Proof of Theorem \ref{thm1}}
\begin{proof}
It follows from \eqref{pw0:improv}, \eqref{pwKG:improv},
\eqref{energy} and Gronwall's inequality with $N\ge21$ that there are two constants $C_1,C_2>1$ such that
\begin{equation*}
E_{2N}[u](t)\le C_1(\ve+\ve_1^2)(1+t)^{C_2\ve_1}
\end{equation*}
and
\begin{align*}
&\sum_{|a|\le N+3}|P_{=0}Z^au(t,x,y)|\le C_1(\ve+\ve_1^2) \w{t+|x|}^{-1/2}\w{t-|x|}^{-0.4},\\
&\sum_{|a|\le N+2}|P_{=0}\p Z^au(t,x,y)|\le C_1(\ve+\ve_1^2)\w{x}^{-1/2}\w{t-|x|}^{-1.4},\\
&\sum_{|a|\le2N-16}|P_{\neq0}Z^au(t,x,y)|\le C_1(\ve+\ve_1^2)\w{t+|x|}^{-1}.
\end{align*}
Let $\ve_1=4C_1\ve$ and $\ve_0=\min\{\frac{1}{16C_1^2},\frac{1}{400C_1C_2}\}$. Then \eqref{BA1}-\eqref{BA4} are improved to
\begin{equation*}
E_{2N}[u](t)\le\frac12\ve_1(1+t)^{\ve_2}
\end{equation*}
and
\begin{equation*}
\begin{split}
&\sum_{|a|\le N+3}|P_{=0}Z^au(t,x,y)|\le\frac12\ve_1 \w{t+|x|}^{-1/2}\w{t-|x|}^{-0.4},\\
&\sum_{|a|\le N+2}|P_{=0}\p Z^au(t,x,y)|\le\frac12\ve_1 \w{x}^{-1/2}\w{t-|x|}^{-1.4},\\
&\sum_{|a|\le2N-16}|P_{\neq0}Z^au(t,x,y)|\le\frac12\ve_1 \w{t+|x|}^{-1}.
\end{split}
\end{equation*}
This, together with the local existence of classical solution to \eqref{QWE}, yields that \eqref{QWE} with \eqref{sym:condition}
admits a unique global solution $u\in\bigcap\limits_{j=0}^{2N+1}C^{j}([0,\infty), H^{2N+1-j}(\R^2\times\T)))$.
Moreover, \eqref{thm1:decay} can be achieved by \eqref{proj:property} and \eqref{BA2}-\eqref{BA4}.

Finally, we prove \eqref{thm1:scater}. We will construct $\{u^\infty_{(0)},u^\infty_{(1)},u^\infty\}$ 
and subsequently show \eqref{thm1:scater}.
Denote $\ds\Lambda:=\Big(-\sum_{j=1}^3\p^2_j\Big)^\frac12$ and $\cU:=(\p_t+i\Lambda)u$.
Then 
\begin{equation*}
(\p_t-i\Lambda)\cU=\Box u.
\end{equation*}
This, together with Duhamel's principle, leads to
\begin{equation}\label{scater:pf5}
\cU(t)=e^{it\Lambda}\cU(0)+\int_0^te^{i(t-s)\Lambda}\Box u(s)ds.
\end{equation}
Set
\begin{equation}\label{scater:pf6}
\begin{split}
&\cU_{(0)}^\infty:=\cU(0)+\int_0^\infty e^{-is\Lambda}\Box u(s)ds,\\
&u^\infty_{(0)}:=\Lambda^{-1}\mathrm{Im}~\cU_{(0)}^\infty,\quad u^\infty_{(1)}:=\mathrm{Re}~\cU_{(0)}^\infty.
\end{split}
\end{equation}
By \eqref{eqn:high}, \eqref{null:high}, the first and fourth lines of \eqref{null:structure}, \eqref{BA2}-\eqref{BA4} and \eqref{pw0:good:low}, we have
that for $0\leq \emph{l} \leq N+2$,
\begin{equation}\label{aux:energy2;Scatt}
\begin{split}
|\Box Z^{\leq \emph{l}} u|&\ls\sum_{|b|+|c|\le \emph{l}}\w{t+|x|}^{-1}(|\p^{\le1}\p Z^bu||ZZ^cu|+|Z\p^{\le1}Z^bu||\p Z^cu|)\\
&+\sum_{|b|+|c|+|d|\le \emph{l}}|P_{=0}\bar\p Z^bu||P_{=0}\p Z^cu||P_{=0}\p Z^du|
+\sum_{\substack{|b|+|c|+|d|\le \emph{l},\\ \mu,\nu\in\{=0,\neq0\}}}|P_{\neq0}\p Z^bu||P_{\mu}\p Z^cu||P_{\nu}\p Z^du|\\
&\ls \ve\w{t+|x|}^{-3/2}(|\p Z^{\leq \emph{l}+1}u|+|P_{=0}\p Z^{\leq \emph{l}}u|+|P_{\neq0}\p Z^{\leq \emph{l}}u|).
\end{split}
\end{equation}
On the other hand, by the standard energy estimates, \eqref{initial:data}, \eqref{proj:property}, \eqref{BA1} and \eqref{aux:energy2;Scatt}, one has
\begin{equation*}
\begin{split}
\|\p Z^{\leq N+2}u(t)\|^2_{L^2_{x,y}}&\ls\|\p Z^{\leq N+2}u(0)\|^2_{L^2_{x,y}}+\int_0^t\|\p_tZ^{\leq N+2}u(s)\|_{L^2_{x,y}}\|\Box Z^{\leq N+2}u(s)\|_{L^2_{x,y}}ds\\
&\ls \ve\|\p Z^{\leq N+2}u(0)\|_{L^2_{x,y}}+\ve^2\int_0^t(1+s)^{-3/2+\ve_2}ds\sup_{s\in[0,t]}\|\p Z^{\leq N+2}u(s)\|_{L^2_{x,y}}\\
&\ls (\ve+\ve^2) \sup_{s\in[0,t]}\|\p Z^{\leq N+2}u(s)\|_{L^2_{x,y}},
\end{split}
\end{equation*}
which further implies
\begin{equation}\label{scater:pf2-Tao-01}
E_{N+2}[u](t)\ls\ve.
\end{equation}
Then, it follows from \eqref{initial:data}, \eqref{scater:pf6}-\eqref{scater:pf2-Tao-01}, the Minkowski inequality and the unitary of $e^{-is\Lambda}$ that
\begin{equation*}
\begin{split}
\|\cU_{(0)}^\infty\|_{H^{N+1}(\R^2\times\T)}&\ls\|\cU(0)\|_{H^{N+1}(\R^2\times\T)}+\int_0^\infty\|e^{-is\Lambda}\Box u(s)\|_{H^{N+1}(\R^2\times\T)}ds\\
&\ls\|\cU(0)\|_{H^{N+1}(\R^2\times\T)}+\int_0^\infty\|\Box \p^{\leq N+1} u(s)\|_{L^2_{x,y}}ds\\
&\ls\ve,
\end{split}
\end{equation*}
which means $(u_{(0)}^\infty,u_{(1)}^\infty)\in \dot H^{N+2}(\R^2\times\T)\times H^{N+1}(\R^2\times\T)$.

Define $u^\infty(t,x)=\Lambda^{-1}\mathrm{Im}e^{it\Lambda}\cU_{(0)}^\infty$. Then $u^\infty$ is the solution
of $\Box u^\infty=0$ with the initial data $(u_{(0)}^\infty,u_{(1)}^\infty)$ at $t=0$.
It follows from \eqref{scater:pf5}-\eqref{scater:pf2-Tao-01} that
\begin{equation*}\label{scater:pf7}
\begin{split}
&\|\p (u^\infty-u(t))\|_{H^{N+1}(\R^2\times\T)}\ls\|e^{it\Lambda}\cU_{(0)}^\infty-\cU(t)\|_{H^{N+1}(\R^2\times\T)}\\
&\ls\int_t^\infty\|e^{i(t-s)\Lambda}\Box u(s)\|_{H^{N+1}(\R^2\times\T)}ds
\ls\int_t^\infty \ve(1+s)^{-3/2}E_{N+2}[u](s)ds 
\ls\ve^2\w{t}^{-1/2},
\end{split}
\end{equation*}
which yields \eqref{thm1:scater}.
\end{proof}

\vskip 0.2 true cm
{\bf \color{blue}{Conflict of interest}}
\vskip 0.2 true cm

{\bf On behalf of all authors, the corresponding author states that there is no conflict of interest.}

\vskip 0.2 true cm

{\bf \color{blue}{Data availability}}

\vskip 0.2 true cm

{\bf Data sharing is not applicable to this article as no new data were created.}

\vskip 0.2 true cm

\end{document}